\documentclass[12pt]{amsart}
\usepackage[normalem]{ulem}
\usepackage{environ, xcolor}
\NewEnviron{commentA}{}
\newcommand\Aon{\RenewEnviron{commentA}{\color{purple}\BODY}}

\Aon{} 

\usepackage{amsmath, amssymb}
\usepackage{array}
\usepackage[frame,cmtip,arrow,matrix,line,graph,curve]{xy}
\usepackage{graphpap, color, paralist, pstricks}
\usepackage[mathscr]{eucal}
\usepackage[pdftex]{graphicx}
\usepackage[pdftex,colorlinks,backref=page,citecolor=blue]{hyperref}
\usepackage{cleveref}
\usepackage{tikz-cd, verbatim}
\usepackage{colonequals}

\usepackage{setspace}

\newtheorem{theorem}{Theorem}[section]
\newtheorem{proposition}[theorem]{Proposition}

\newtheorem{lemma}[theorem]{Lemma}

\theoremstyle{definition}

\newtheorem{example}[theorem]{Example}

\newtheorem{question}[theorem]{Question}

\usepackage[margin=1in]{geometry} 
\usepackage{amsmath,amssymb,mathtools, mathrsfs,esint,tikz-cd,verbatim}

\newcommand{\Z}{{\mathbb Z}}
\newcommand{\Q}{{\mathbb Q}}

\newcommand{\C}{{\mathbb C}}
\newcommand{\A}{{\mathbb A}}

\newcommand{\M}{{\mathcal M}}
\newcommand{\D}{{\mathscr{D}}}
\renewcommand{\P}{{\mathbb P}}
\renewcommand{\O}{{\mathcal O}}
\renewcommand{\H}{{\mathcal H}}

\newcommand{\on}[1]{\operatorname{#1}}

\newcommand{\Aut}{{\on{Aut}}}

\newcommand{\id}{{\on{id}}}

\newcommand{\rRepl}{\on{rRepl}}

\DeclareMathOperator{\Spec}{\mathrm{Spec}}
\DeclareMathOperator{\Bir}{\mathrm{Bir}}

\DeclareMathOperator{\bir}{\simeq_{\mathrm{bir}}}

\DeclareMathOperator{\Aff}{\mathrm{Aff}}
\DeclareMathOperator{\Tr}{\mathrm{Tr}}
\DeclareMathOperator{\SL}{\mathrm{SL}}

\newcommand{\Set}[2]{\left\{\,#1 \ \middle| \ #2\,\right\}}


\title{Characterizing Varieties Using Birational Transformations}

\author{Nathan Chen}
\address{Department of Mathematics, Harvard University, 1 Oxford Street, Cambridge, MA 02138}
\email{nathanchen@math.harvard.edu}

\author{Louis Esser}
\address{Department of Mathematics, Princeton University, Fine Hall, Washington Road, Princeton, NJ 08544-1000, USA}
\email{esserl@math.princeton.edu}

\author{Andriy Regeta}
\address{Dipartimento di Matematica,
Universit\`a di Padova, Via Trieste 63, I-35121 Padova}
\email{andriyregeta@gmail.com}

\author{Christian Urech}
\address{Department of Mathematics, ETH Zurich, 8092 Zurich, Switzerland}
\email{christian.urech@math.ethz.ch}

\author{Immanuel van Santen}
\address{Mathematisches Institut, Universit\"at Bern, Sidlerstrasse 5, 3012 Bern, Switzerland}
\email{immanuel.van.santen@math.ch}

\begin{document}

\begin{abstract}
Suppose $X$ is an irreducible complex variety.  We show that when $X$ is ruled,
the group of birational transformations $\Bir(X)$, as a group, determines $X$
up to birational transformations and automorphisms of the base field.  
In contrast, we demonstrate that this same property never holds for non-uniruled
varieties.
\end{abstract}

\maketitle

\begin{spacing}{0}
\tableofcontents
\end{spacing}

\section{Introduction}

It's often the case that the group of symmetries of a mathematical object 
determines this object, up to isomorphism.  For instance, this is true in the 
setting of differential manifolds \cite{Filipkiewicz} and compact topological manifolds \cite{Whittaker63}.

For algebraic varieties $X$ the situation is different as ``algebraic symmetries" are usually too rigid to determine $X$;
indeed, most varieties have a trivial group of symmetries. However, if an algebraic variety has many symmetries, the situation changes. For example, in the category of affine varieties, if $X \simeq \mathbb{A}^n$, then $X$ is determined by its automorphism group \cite{CaReXi2023Families-of-commut} (see also \cite{RvS25}).
The motivating question behind this paper is the following:
\begin{question}
\label{quest:main}
To what extent does the group of birational transformations $\Bir(X)$ of a variety $X$ determine $X$?
\end{question}
Previous work of the last three authors shows that for complex varieties,
the group of birational 
transformations characterizes rationality and ruledness \cite{RUvS24b} 
(see also \cite{Ca2014Morphisms-between-}). More
precisely, given an arbitrary irreducible variety $X$, the existence of a group isomorphism
$\Bir(X) \cong \Bir(\P^n)$ implies that $X$ is birational to $\P^n$; similarly,
the existence of
a group isomorphism $\Bir(X) \cong \Bir(Y \times \P^1)$ for some other variety $Y$
implies that $X$ is ruled, i.e., birational to $Z \times \P^1$ for some variety
$Z$.  However, this statement does not provide any information about $Z$. 

In this paper, we strengthen this result by showing that for any ruled 
complex variety $X$, the group $\Bir(X)$ determines $X$, up to birational 
transformations and automorphisms of the base field $\C$. Our main result is the following:
\begin{theorem}
\label{thm:intromain}
Let $X$ and $Y$ be irreducible
varieties over $\C$ with the property that there is
a group isomorphism $\Bir(X) \cong \Bir(\mathbb{A}^s \times Y)$ for some positive integer $s$. Then $X$ is birational to $\A^{s} \times Z$ for some irreducible variety $Z$ which is birational to $Y$, up to base change by a field automorphism of $\C$.
\end{theorem}
\noindent In particular, some birational models of $X$ and $\mathbb{A}^s \times Y$ are 
isomorphic as schemes (over $\Z$). A key part of the proof is to show that 
isomorphisms of birational transformation groups
carry unipotent elements to unipotent elements (cf. Lemma~\ref{lem:unip}).

The conclusion above is the best we can hope for in this generality.
Indeed, if $V$ is an 
irreducible variety, $\tau: \C \rightarrow \C$ is a field isomorphism, and 
$V^{\tau}$ is the variety obtained after a base change of $V$ by $\tau$, 
then $\tau$ naturally induces a group isomorphism between $\Bir(V^{\tau})$ and 
$\Bir(V)$. In other words, we cannot distinguish $V^{\tau}$ and $V$
using only the group structure 
of their groups of birational transformations. In the literature, the two varieties $V, V^{\tau}$ are said to be \textit{conjugate}. We note that conjugate varieties can look quite different. Although their Betti numbers must be the same, as Serre showed that these are algebraic invariants for smooth projective varieties, other topological 
invariants may differ.  Indeed, Serre \cite{Serre} exhibited pairs of conjugate smooth varieties that
have non-isomorphic fundamental groups; many other related examples have since been discovered
(see, e.g., \cite{Abelson}, \cite{ABCRCA}, \cite{Shimada09}, etc.). Hence $V$ and
$V^{\tau}$ may not even be smoothly deformation equivalent.

In contrast to \Cref{thm:intromain}, when a complex variety $X$ is not covered by rational curves, 
the group structure of $\Bir(X)$
is \textit{never} enough to determine the scheme $X$, or even its dimension.
More precisely, we show:
\begin{theorem}
\label{thm:introBirX=BirXxC}
Let $X$ be an irreducible non-uniruled variety of dimension $n \geq 1$, and
let $C$ be a very general curve of genus $g \geq \max \{ 3, n+1 \}$. Then $\Bir(X) \cong \Bir(X \times C)$.
\end{theorem}
\noindent The situation in the intermediate case, where $X$ is uniruled but not ruled, is 
much less clear.  Since these varieties do not carry faithful birational 
actions of 
positive-dimensional linear algebraic groups \cite[Corollary 1]{Matsumura63}, the methods behind the proof of
\Cref{thm:intromain} do not apply.  Nevertheless, some varieties of this type,
such 
as cubic hypersurfaces or more generally conic bundles, can have birational transformation groups that are quite large.  We
illustrate in \Cref{sect:nonruled} the range of behavior that can occur in this 
setting with various examples.

Throughout the paper, we work over the complex numbers $\C$.  The same results 
hold for any uncountable algebraically closed field $k$ of characteristic zero. A \textit{variety} is a separated integral scheme of finite type over $k$.

\subsection*{Acknowledgements}

We would like to thank Rob Lazarsfeld for helpful discussions. During the preparation of this article, C.U. was partially supported by the SNSF grant 10004735 and I.v.S. was supported by the SNSF grant 10000940.

\section{Birational transformations of ruled varieties}
\label{sect:mainthm}

The goal of this section is to prove \Cref{thm:intromain}.
The proof will crucially use the study of unipotent elements
of birational transformation groups from \cite[Section 6]{RUvS24b}.

\subsection{Preliminaries}

Our first observation is that every field isomorphism
between function fields of complex varieties can be decomposed into an automorphism of the 
base field $\C$ and an isomorphism of function fields over $\C$. In other words,
if $V_1$ and $V_2$ are varieties such that $\C(V_1)$ and $\C(V_2)$ are
isomorphic as fields, then $V_1$ and $V_2$ are birational after a base change. 

\begin{lemma}
\label{lem:field_iso}
	Let $k$ be an algebraically closed field, and let $k(V_1)$ and $k(V_2)$ be the function fields of two varieties $V_{1}$ and $V_{2}$, respectively. Suppose there exists a field isomorphism $\varphi\colon k(V_1) \to k(V_2)$. Then there exists a field automorphism $\tau$ of $k$ and a $k$-automorphism
    $\psi \colon k(V_1^{\tau}) \to k(V_2)$ where $V_1^\tau \to \Spec(k)$ 
    is the pullback of $V_1 \to \Spec(k)$ via the automorphism of $\Spec(k)$ induced by $\tau^{-1}$.
\end{lemma}

\begin{proof}
	Without loss of generality, by passing to their normalizations we may assume that both $V_{1}$ and $V_{2}$ are normal.
    Furthermore, taking completions, we may assume that $V_1$, $V_2$ are complete.
    Consider the set of (infinitely) divisible elements in $k(V_{1})$:
    \[ \D(V_{1}) \coloneqq \bigcap_{n \geq 1} \{z^n\mid z \in k(V_1)\}. \]
    We claim that $\D(V_{1}) = k$. It is clear that $k$ is contained in $\D(V_{1})$ since $k$ is algebraically closed. For the other direction, suppose $f \in \D(V_{1})$ is a nonzero element, i.e., for every $n \geq 1$ there exists $g_{n} \in k(V_1)$ such that $g_{n}^{n} = f$. Note that every prime (Weil) divisor $D \subset V_{1}$ gives rise to a discrete valuation $\nu_{D}$ on $k(V_{1})$ and
    \[ \nu_{D}(f) = \nu_{D}(g_{n}^{n}) = n \cdot \nu_{D}(g_{n}). \]
    Since $\nu_{D}(f)$ is divisible by $n$ for every $n \geq 1$, this implies that $\nu_{D}(f) = 0$ for all $D$. Thus, $f$ (thought of as a rational function on $V_{1}$) has no zeros or poles anywhere, and hence $f \in k$.
    Finally, repeating the process for $V_{2}$ and noting that $\varphi$ is a field isomorphism, it follows that $\varphi$ sends $\D(V_{1}) = k$ isomorphically to $\D(V_{2}) = k$. If we let $\tau$ denote this automorphism of $k$, then performing the base change of $k(V_{1}) \rightarrow k$ by $\tau^{-1}$ allows us to write the desired map $\psi$ as a composition
    \[ k(V_{1}^{\tau}) = {k(V_{1}) \otimes_{\tau^{-1}} k} \longrightarrow k(V_{1}) \xlongrightarrow{\varphi} k(V_{2}). \qedhere \]
\end{proof}

By the lemma above, in order to prove \Cref{thm:intromain} it suffices to produce an irreducible variety $Z$ such that $X$ is birational to $\mathbb{A}^s \times Z$ and then prove that $\C(Z)$ and $\C(Y)$ are abstractly isomorphic as fields.

We now recall some results about the geometric structure of $\Bir(X)$. An algebraic family of birational transformations of a variety $X$ parametrized by a variety
$V$ is a $V$-birational map $\theta\colon V \times X \to V \times X$ such that $\theta$ induces an isomorphism between open dense subsets of $V \times X$ that both surject to $V$ under the first projection.
Such an algebraic family gives us a map from $V$ to $\Bir(X)$. Maps of this type are called \textit{morphisms}. The \textit{Zariski topology} on $\Bir(X)$ is defined as the finest topology such that all the morphisms are continuous. If $G\subset\Bir(X)$ is a closed subgroup, then its connected component of the identity $G^\circ$ has at most countable index in $G$ (see \cite[Corollary~4.11]{RUvS24a}). A subgroup of $\Bir(X)$ is an \textit{algebraic subgroup} if it is the image of a morphism from an algebraic group that is also a group homomorphism. Algebraic subgroups are always closed and finite dimensional in $\Bir(X)$. On the other hand, every closed and finite dimensional subgroup in $\Bir(X)$ has a unique structure of an algebraic group (see \cite[Corollary~5.12, Theorem~1.3]{RUvS24a} for details and proofs). An element $g\in\Bir(X)$ is \textit{unipotent} if it is contained in a unipotent algebraic subgroup.

Let us also recall the notion of a rational replica (see \cite[Section~6.2]{RUvS24b}). Let $G$ be an algebraic group with a faithful regular action on $X$. Every $G$-invariant rational map $f\colon X\dashrightarrow G$ defines a birational transformation of $\mathrm{dom}(f)$ given by $x\mapsto f(x)\cdot x$. In this way, we obtain a group homomorphism from the group of $G$-invariant rational maps $\mathrm{Rat}^G(X,G)$ to $\Bir(X)$, whose image we denote by $\mathrm{rRepl}_X(G)$. If $G\subset\Bir(X)$ is an algebraic subgroup, we first choose -- using Weil's regularization theorem -- a $G$-equivariant birational map $\varphi\colon X\dashrightarrow X'$, such that the conjugate action of $G$ on $X'$ is regular, and then define $\mathrm{rRepl}(G):=\varphi^{-1}\mathrm{rRepl}_{X'}(G)\varphi$. This construction does not depend on the choice of $\varphi$ or $X'$. A useful property of the subgroup $\mathrm{rRepl}(G)$ is that it coincides with its centralizer in $\Bir(X)$ \cite[Proposition~6.7]{RUvS24b}. Finally, given a dominant morphism $\pi: X \rightarrow Y$, we will use $\Bir(X, \pi)$ to denote the subgroup of elements preserving the fibers of $\pi$ and $\Bir(X/Y)$ to denote the subgroup of elements inducing the identity on $Y$.

The following result shows that group isomorphisms of birational transformation
groups always preserve unipotent elements. 

\begin{lemma}
\label{lem:unip}
Let $V_1$ and $V_2$ be varieties, and suppose there is a group isomorphism
$\Xi: \Bir(V_1) \rightarrow \Bir(V_2)$. 
If $u \in \Bir(V_1)$ is a unipotent element, then
$u' \coloneqq \Xi(u)$ is unipotent.
\end{lemma}

\begin{proof}
Up to birational equivalence, we may replace $V_1$ by 
$\mathbb{A}^1 \times W_1$, where $u$ now acts on $\mathbb{A}^1$ by 
translation and on $W_1$ by the identity.  Denote by 
$U \subset \Bir(\A^1 \times W_1)$ the algebraic group of transformations
$(p,w) \mapsto (p + c,w)$ for $c \in \mathbb{G}_{\on{a}}$, which 
contains $u$, and $D \subset \Bir(\A^1 \times W_1)$ the algebraic group
of transformations $(p,w) \mapsto (ap,w)$ for $a \in \mathbb{G}_{\on{m}}$.
Then $\rRepl(D)$ acts on $\rRepl(U)$ by conjugation with two orbits
(one is the identity and the other consists of all non-trivial elements), and
both these groups are self-centralizing.  The images of these two groups
are self-centralizing, hence closed, in $\Bir(V_2)$, so the connected
components of the identity have countable index.  Thus,
$\Xi(\rRepl(D))^{\circ}$ acts on $\Xi(\rRepl(U))^{\circ}$ with countably
many orbits.  Therefore, \cite[Theorem 6.1]{RUvS24b} shows that there
exists a unipotent element different from the identity in $\Bir(V_2)$
that commutes with $\Xi(\rRepl(U))^{\circ}$.

Now, let $H \subset \rRepl(U)$ be a countable index subgroup. 
By~\cite[Lemma~5.5]{RUvS24b}, $H$ is dense in $\rRepl(U)$,
so
$\mathrm{Cent}_{\Bir(\A^1 \times W_1)}(H) 
 = \mathrm{Cent}_{\Bir(\A^1 \times W_1)}(\rRepl(U)) = \rRepl(U)$. 
Since there is a unipotent element in $\Bir(V_2)$ commuting with 
$\Xi(\rRepl(U))^{\circ}$, the above argument shows it commutes with
all of $\Xi(\rRepl(U))$ and hence belongs to $\Xi(\rRepl(U))$.
We already saw that $\rRepl(D)$ acts transitively by conjugation on the 
set of non-trivial elements of $\rRepl(U)$, so all non-trivial elements in
$\Xi(\rRepl(U))$ are also conjugate.  This means that every element
of $\Xi(\rRepl(U))$ is unipotent.  In particular, $u' = \Xi(u)$ is 
unipotent.
\end{proof}

\subsection{Proof of \Cref{thm:intromain}}

Building off of \Cref{lem:unip}, we will ultimately show as part of 
the proof of \Cref{thm:intromain}
that 
a group of translations on $\mathbb{A}^s \times Y$ must map to a group
of translations on $X$ of the same dimension under 
$\Xi: \Bir(\mathbb{A}^s \times Y) \cong \Bir(X)$.

For brevity, we write $\mathrm{Tr}_s(\C(Y))$ for the rational replica
of the algebraic group given by elements of $\Bir(\mathbb{A}^s \times Y)$ that
act by translations on $\mathbb{A}^s$ and the identity on $Y$.  We adopt a similar
convention for rational replicas of other algebraic groups acting on $\mathbb{A}^s$,
e.g., $\mathrm{Aff}_s$ and $\SL_s$.
We can and will assume that $s$ is maximal, meaning that
$\mathbb{A}^s \times Y$ is not birational to any other variety of the form 
$\mathbb{A}^{s'} \times Y'$ with $s' > s$.  One consequence of this is that
$\Bir(Y)$ does not
contain any nontrivial unipotent element.

We will need two preliminary lemmas.  The first deals with the normalizer
of the translation group in the birational transformation group of a variety
of the form $\mathbb{A}^s \times V$.

\begin{lemma}
\label{lem:normalizer}
Let $V$ be a variety and
$\Tr_s(\C(V)) \subset \Bir(\A^s \times V)$ the subgroup of rational 
families of translations of $\A^s \times V$ over $V$.  Then the normalizer
of $\Tr_s(\C(V))$ in $\Bir(\A^s\times V)$ is equal to $\Aff_s(\C(V))\rtimes \Bir(V)$,
where $\Aff_s(\C(V))$ denotes the group of affine transformations over the function field $\C(V)$. 
\end{lemma}
	
\begin{proof}
Let $f\in \mathrm{Norm}_{\Bir(\A^s\times V)}(\Tr_s(\C(V)))$. 
Then $f$ permutes the orbits of $\Tr_s(\C(V))$, hence it preserves
the fibration given by the projection to $V$, so it induces a rational
action on $V$. The group of elements in $\mathrm{Norm}(\Tr_s(\C(V)))$ 
that act on $V$ as the identity is exactly $\mathrm{Aff}_s(\C(V))$ (see 
\cite[Lemma 6.11]{RUvS24b}).
The claim follows.
\end{proof}

Next, we examine the structure of the subgroup generated by unipotent elements
inside this normalizer.

\begin{lemma}
\label{lem:unip_in_norm}
Let $V$ be a variety and 
$\Aff_s(\C(V))\rtimes \Bir(V) \subset \Bir(\A^s \times V)$ the normalizer
from \Cref{lem:normalizer}.  Then the 
subgroup $G$ generated by unipotent elements in this normalizer contains
$\SL_s(\C(V)) \ltimes \Tr_s(\C(V)) \subset \Aff_s(\C(V))$. 
In particular, if $\Bir(V)$ contains
no unipotent elements, then $G$ equals $\SL_s(\C(V)) \ltimes \Tr_s(\C(V))$.
\end{lemma}

\begin{proof}
We first prove that $\SL_s(\C(V)) \ltimes \Tr_s(\C(V))$ is the
subgroup generated by unipotent elements inside $\Aff_s(\C(V))$. 
But in light of \cite[Proposition 4.2]{RUvS24b}, the unipotents in 
$\mathrm{Aff}_s(\C(V)) \subset \Bir(\mathbb{A}^s \times V)$ are exactly those
elements which are unipotent in $\mathrm{Aff}_s(\C(V))$ considered as a linear
algebraic group over $\C(V)$.  Then the claim follows from the well-known
fact that $\SL_s(\C(V))$ is generated by matrices of the form $I + a E_{i,j}$,
with $i \neq j$, plus the observation that all unipotent matrices must have 
determinant $1$.

The first statement about the subgroup generated by unipotent elements in the
larger group $\Aff_s(\C(V))\rtimes \Bir(V)$ is now clear.  If, furthermore,
$\Bir(V)$ contains no nontrivial unipotent elements, then any unipotent $u$ in the 
semidirect product must map to the identity in $\Bir(V)$.

Indeed, we claim that the image of $u$ in $\Bir(V)$ is unipotent.  This follows
from the fact that the homomorphism $\Bir(\A^s \times V,\pi_2) \rightarrow \Bir(V)$
preserves algebraic families parametrized by normal varieties
\cite[Proposition 7.5]{RUvS24a}, so that the $\mathbb{G}_{\on{a}}$-action induced
by $u$ descends to a $\mathbb{G}_{\on{a}}$-action on $V$ generated by the image of $u$.
Therefore, the image of $u$ must be trivial.  This completes the proof.
\end{proof}

Now we are ready to prove the main theorem.

\begin{proof}[Proof of \Cref{thm:intromain}]

We first note that $\mathrm{Tr}_s(\C(Y))$ is a maximal commutative subgroup of
unipotent elements in $\Bir(\mathbb{A}^s \times Y)$.  That is, among the class
of commutative subgroups consisting only of unipotents, it is a maximal
element under inclusion.  This follows from \cite[Proposition 6.7]{RUvS24b},
which shows that $\mathrm{Tr}_s(\C(Y))$ is self-centralizing: any element that
commutes with all elements in $\mathrm{Tr}_s(\C(Y))$ must already be in
$\mathrm{Tr}_s(\C(Y))$.

By \Cref{lem:unip}, the image $\Xi(\mathrm{Tr}_s(\C(Y)))$
is also a maximal commutative subgroup of unipotent elements in $\Bir(X)$.
A finite collection of commuting unipotent elements generates an algebraic 
subgroup, so we may choose an algebraic subgroup 
$U \subset \Xi(\mathrm{Tr}_s(\C(Y)))$ with 
$\C(X)^U = \C(X)^{\Xi(\mathrm{Tr}_s(\C(Y)))}$; moreover, we can arrange so that
the dimension of the general orbit of $U$ equals $\dim(U)$ \cite[Lemma 6.12]{RUvS24b}. By \cite[Theorem 2]{Brion21}, there
exists some irreducible variety $Z$ such that $X \bir \mathbb{A}^r \times Z$,
where $r = \dim(U)$,
and where \\ $\Xi(\mathrm{Tr}_s(\C(Y))) = \mathrm{Tr}_r(\C(Z))$.
Using \Cref{lem:normalizer}, we have that $\Xi$ induces an isomorphism of normalizers
\[ \mathrm{Aff}_s(\C(Y)) \rtimes \Bir(Y) \xrightarrow{\Xi}
\mathrm{Aff}_r(\C(Z)) \rtimes \Bir(Z). \]
By assumption, $s$ is maximal so \Cref{lem:unip_in_norm} shows that the subgroup generated by
unipotent elements in $\mathrm{Aff}_s(\C(Y)) \rtimes \Bir(Y)$
is the subgroup $\mathrm{SL}_s(\C(Y)) \ltimes \Tr_s(\C(Y)) \subset \mathrm{Aff}_s(\C(Y))$. 

Let $G \subset \mathrm{Aff}_r(\C(Z)) \rtimes \Bir(Z)$ be the subgroup generated
by unipotent elements.  A priori, it could be that $\Bir(Z)$ contains
unipotents, so that $G$ is larger than
$\mathrm{SL}_r(\C(Z)) \ltimes \Tr_r(\C(Z))$.  We will show that this in fact does not occur.
Since $\Xi$ preserves unipotent elements by 
\Cref{lem:unip}, it must preserve the subgroup 
generated by unipotents in the normalizer of the translation group.  Hence
we have that $\Xi$ restricts to
an isomorphism
\[ \mathrm{SL}_s(\C(Y)) \ltimes \Tr_s(\C(Y)) \xrightarrow{\Xi} G. \]
Because $\Xi(\Tr_s(\C(Y))) = \Tr_r(\C(Z))$ and these are normal subgroups
of the left and right sides, respectively, the isomorphism descends to quotients:
\begin{equation}
\label{eq:unip_iso}
    \mathrm{SL}_s(\C(Y)) \xrightarrow{\cong} G/\Tr_r(\C(Z)).
\end{equation}

\noindent \textbf{Case 1}: $s > 1$

Since $s > 1$, the group $\mathrm{SL}_s(\C(Y))$ on the left in \eqref{eq:unip_iso}
is almost simple; in particular, the only proper normal
subgroups are finite and central in $\mathrm{SL}_s(\C(Y))$ \cite[p. 168]{Humphreys}.
Therefore, this property holds for the group
on the right as well.  Denote by $\Bir(Z)_u$ the subgroup generated by all unipotent elements of 
$\Bir(Z)$. By construction, $\Bir(Z)_u$ is contained in $G$ and hence 
the identity on $\Bir(Z)_u$ splits into
\[ \Bir(Z)_u \rightarrow G/\Tr_r(\C(Z)) \rightarrow \Bir(Z)_u \ . \]
Suppose for contradiction that $\Bir(Z)_u$ is non-trivial. Then the kernel $K$ of
\[ G/\Tr_r(\C(Z)) \rightarrow \Bir(Z)_u \]
is a proper normal subgroup, hence it is finite and central. If $r > 1$, then $K$ contains a nontrivial special linear group and must be infinite, a contradiction.  If $r = 1$, then $K$ is trivial and we have that $G/\Tr_1(\C(Z)) \rightarrow \Bir(Z)_u$ is an isomorphism.
The
group $\mathrm{SL}_s(\C(Y)) \ltimes \Tr_s(\C(Y))$ acts by conjugation
on $\Tr_s(\C(Y))$ with two orbits. 
Hence, this is satisfied as well for the conjugation action of $G$ on $\Tr_1(\C(Z))$. This is a contradiction, since 
the subgroup of ``constant" translations $\C^+ \subset \C(Z)^+ = \Tr_1(\C(Z))$ is pointwise fixed under the action by
$G = \Tr_1(\C(Z)) \rtimes \Bir(Z)_u$. 

This argument shows that $\Bir(Z)$ contains no nontrivial unipotents.
In particular, $G = \SL_r(\C(Z)) \ltimes \Tr_r(\C(Z))$. Also we have $r > 1$, because 
otherwise $G$ is abelian.

In summary, we must have that both $r,s > 1$ and
$G = \mathrm{SL}_r(\C(Z)) \ltimes \Tr_r(\C(Z))$,
so that the isomorphism \eqref{eq:unip_iso} is between
$\mathrm{SL}_s(\C(Y))$ and $\mathrm{SL}_r(\C(Z))$.
Then \cite[Theorem 5.6.10]{O'Meara} implies that $r = s$ and the fields
$\C(Y)$ and $\C(Z)$ are isomorphic, completing the proof of the main theorem
in this case.

\noindent \textbf{Case 2}: $s = 1$

In this case, we still have that the translation group $\C(Y)^+ \coloneqq \Tr_1(\C(Y))$
maps to a translation group $\Tr_r(\C(Z))$. 
But then we must have $r = 1$, or else
the proof in Case $1$, applied in reverse, would show that 
$\mathbb{A}^1 \times Y$ actually splits off some $\mathbb{A}^r$ with $r > 1$,
contradicting the maximality of $s$.  Hence we have a birational equivalence
$X \bir \mathbb{A}^1 \times Z$, the map $\Xi$ sends $\C(Y)^+$ to 
$\C(Z)^+$, and moreover $\Bir(Z)$ does not
contain any unipotents by the same reasoning as above.  However, since
$\mathrm{SL}_1$ is trivial, we must employ a different strategy to obtain the
isomorphism of function fields.

We fix some notation. Let $V$ be an irreducible variety. 
Inside $\mathrm{PGL}_2(\C(V)) = \Bir(\A^1 \times V/V)$ we identify  $\Aff_1(\C(V)) = \Bir(\A^1 \times V/V)$ with the subgroup of upper triangular matrices, 
$\C(V)^+$ with the subgroup of strictly upper triangular matrices 
and $\C(V)^*$ with the subgroup of diagonal matrices. Moreover, we consider the following involution:
\begin{equation}\label{eq:winvolution}
    w \coloneqq \begin{pmatrix} 0 & 1 \\ 1 & 0 \end{pmatrix}
    \in \mathrm{PGL}_2(\C(V))
    \, . 
\end{equation}
\noindent For the proof, we will use that $\mathrm{PGL}_2(\C(V))$ is generated by $\Aff_1(\C(V))$ and
$w$. Moreover, we will use that the normalizer of $\C(V)^\ast$ in 
$\mathrm{PGL}_2(\C(V))$ is generated by $w$ modulo elements from $\C(V)^\ast$.

The intermediate goal is to show that 
$\Xi(\mathrm{PGL}_2(\C(Y))) = \mathrm{PGL}_2(\C(Z))$. In order to achieve this, 
we require the following lemma.

\begin{lemma}
\label{lem:doubling_action}
Let $V$ be a variety and $G \cong \C(V)^+$ the subgroup
of translations inside 
$\Bir(\A^1 \times V)$.  Suppose that $\alpha$ is an element of $\Bir(\A^{1}\times V)$
with the property that $\alpha g \alpha^{-1} = g^2$ for all $g \in G$.  Then, up to a translation, 
$$\alpha = \begin{pmatrix}
    2 & 0 \\ 0 & 1
\end{pmatrix} \in \mathrm{PGL}_2(\C(V)).$$
\end{lemma}

\begin{proof}
Since conjugation by $\alpha$ sends $G$ to itself by assumption, 
we have that
\[ \alpha \in \mathrm{Norm}_{\Bir(\A^1 \times V)}(G) = 
(\C(V)^* \ltimes \C(V)^+) \rtimes \Bir(V). \]
The action of $\alpha$ on the translation group $G = \C(V)^+$ depends only on the image of $\alpha$ in the quotient
$\C(V)^* \rtimes \Bir(V)$, so after composing with a translation we may assume that $\alpha = \alpha_1 \alpha_2$, where $\alpha_1 \in \C(V)^*$ and $\alpha_2 \in \Bir(V)$.

First consider a nontrivial constant translation $g \in \C^+ \subset \C(V)^+$, which shifts all fibers of $\mathbb{A}^1 \times V \rightarrow V$ by $g$. Conjugation by elements of $\Bir(V)$ is trivial on this element so
$$\alpha g \alpha^{-1} = \alpha_1 \alpha_2 g \alpha_2^{-1} \alpha_1^{-1}
 = \alpha_1 g \alpha_1^{-1},$$
 which must equal $g^2$ (i.e., a shift by $2g \in \C^+$).  The only element
$\alpha_1 \in \C(V)^*$ that accomplishes this is
$$\alpha_1 = \begin{pmatrix}
    2 & 0 \\ 0 & 1
\end{pmatrix} \in \mathrm{PGL}_2(\C(V)),$$
where once again we have identified the affine transformations as a subset of the
projective transformations $\mathrm{PGL}_2(\C(V))$.

It remains to see that $\alpha_2 = \id$.  Suppose to the contrary that 
$\alpha_2 \neq \id$ and choose $x \in V$ in the domain of $\alpha_2$ with
$\alpha_2(x) = x' \neq x$.  Choose an element $g \in \C(V)^+$ corresponding
to a rational function $g \colon V \dashrightarrow \A^1$
containing both $x,x'$ in its domain
and with the property that $g(x) \neq g(x')$.  Then we have for 
$t \in \mathbb{A}^1$ that
\begin{align*}
(\alpha g \alpha^{-1})(t,x') & = 
\alpha_1 \alpha_2 g \alpha_2^{-1} \alpha_1^{-1}(t,x') =
\alpha_1 \alpha_2 g \alpha_2^{-1}(t/2,x') = \alpha_1 \alpha_2 g (t/2,x) \\
& = \alpha_1 \alpha_2 (t/2 + g(x),x) = \alpha_1 (t/2 + g(x),x') = (t + 2g(x),x').
\end{align*}
But this means that $\alpha g \alpha^{-1}$ acts by translation by $2g(x)$
on the $x'$ fiber, which is not equal to $2g(x')$, a contradiction.  Hence
we have shown that $\alpha = \alpha_1$ has the required form.
\end{proof}

The key takeaway from this lemma is that the action of $\alpha$
on the translation group by conjugation determines it uniquely (up to translations).

Returning to the proof of the main theorem in Case 2, Lemma~\ref{lem:doubling_action} shows that 
$\Xi(\alpha) = \alpha g$ for some translation $g \in \C(Z)^+$. 
By composing $\Xi \colon \Bir(\A^1 \times Y) \to \Bir(\A^1 \times Z)$ with the conjugation 
by a suitable translation from $\C(Z)^+$, we may thus assume without loss of generality that 
$\Xi(\alpha) = \alpha$.

Let $\varphi \in \Bir(\A^1 \times Y)$ satisfy $\varphi \alpha \varphi^{-1} = \alpha^{-1}$.
Then $\varphi$ permutes the closures of the $\langle \alpha \rangle$-orbits in $\A^1 \times Y$, 
that is, it permutes the fibers of the projection $\A^1 \times Y \to Y$. 
Hence $\varphi = \varphi_1 \varphi_2$, where  $\varphi_1 \in \mathrm{PGL}_2(\C(Y))$
and $\varphi_2 \in \Bir(Y)$. Using the fact that $\alpha$ commutes with $\Bir(Y)$,
we have $\varphi_1 \alpha \varphi_1^{-1} = \alpha^{-1}$ inside $\mathrm{PGL}_2(\C(Y))$.
Since $\langle \alpha \rangle$ is dense in the subgroup $\C(Y)^*$ of $\mathrm{PGL}_2(\C(Y))$, 
we see that $\varphi_1$ normalizes $\C(Y)^*$ but doesn't centralize $\C(Y)^*$.
Hence, $\varphi \in w (\C(Y)^* \rtimes \Bir(Y))$ where $w$ is the involution in \eqref{eq:winvolution}.
We conclude that 
\[
   \Set{\varphi \in \Bir(\A^1 \times Y)}{\varphi \alpha \varphi^{-1} = \alpha^{-1}}
   = w (\C(Y)^* \rtimes \Bir(Y) )\, .
\]
The above calculation shows that 
$\Xi(w) = w \psi$ for some $\psi \in \C(Z)^* \rtimes \Bir(Z)$. 
Recall that $\Xi$ maps $\Aff_1(\C(Y)) \rtimes \Bir(Y)$ isomorphically onto
$\Aff_1(\C(Z)) \rtimes \Bir(Z)$, since these are the normalizers of the
translation groups.
Since $\mathrm{PGL}_2(\C(Y))$ is generated by $\Aff_1(\C(Y))$ and $w$, 
we deduce that $\Xi$ maps $\mathrm{PGL}_2(\C(Y))$ into 
$\mathrm{PGL}_2(\C(Z)) \rtimes \Bir(Z)$. 
Moreover, note that $\mathrm{PGL}_2(\C(Z))$ is precisely the subgroup generated 
by unipotent elements inside \\ $\mathrm{PGL}_2(\C(Z)) \rtimes \Bir(Z)$, 
since $\Bir(Z)$ contains no nontrivial unipotents. 
As $\Xi$ preserves unipotents by Lemma~\ref{lem:unip}, we obtain 
$\Xi(\mathrm{PGL}_2(\C(Y))) \subset \mathrm{PGL}_2(\C(Z))$.  Symmetrically,
$\Xi^{-1}(\mathrm{PGL}_2(\C(Z))) \subset \mathrm{PGL}_2(\C(Y))$, so in fact
$\Xi(\mathrm{PGL}_2(\C(Y))) = \mathrm{PGL}_2(\C(Z))$.

Another application of \cite[Theorem 5.6.10]{O'Meara} then shows that
$\C(Y)$ and $\C(Z)$ are abstractly isomorphic as fields.  This completes the
proof of \Cref{thm:intromain} in the case $s = 1$.
\end{proof}

\section{Birational transformations of non-ruled varieties}
\label{sect:nonruled}

\Cref{thm:intromain} shows that the group $\Bir(X)$ characterizes $X$ 
up to birational transformations and
automorphisms of the base field when $X$ is ruled.  This naturally 
suggests the question of what happens for other varieties. As a partial 
answer, we will show that non-uniruled varieties \textit{never} 
have the property above. We will end the section with some examples and questions.

\subsection{Proof of \Cref{thm:introBirX=BirXxC}}

Recall that a variety $X$ is \textit{uniruled} if there exists a 
variety $Y$ and a dominant rational map 
$\mathbb{A}^1 \times Y \dashrightarrow X$ that does not factor
through the second projection.  In other words, $X$ is covered
by rational curves \cite[Remarks~4.2(4)]{Debarre01}. In the non-uniruled case, our goal is to prove
the following more precise formulation of \Cref{thm:introBirX=BirXxC}:

\begin{theorem}\label{thm:BirX=BirXxC}
Let $X$ be an irreducible non-uniruled variety of dimension $n \geq 1$, and let $C$ be a very general curve of genus $g \geq \max \{ 3, n+1 \}$. Then the natural injection $\Bir(X) \to \Bir(X \times C)$ is an isomorphism.
\end{theorem}

Before proving this, let us begin with some notation. For $g \geq 3$, let $\M_{g}^{\circ} \subset \M_{g}$ denote the locus of curves $[C]$ that
\begin{enumerate}
\item satisfy $\Aut(C) = \{ 1 \}$, and
\item do not admit a map to another curve of genus $h > 0$.
\end{enumerate}
Property (1) is satisfied for an open dense subset of $\M_{g}$ by \cite{Popp69}. The Riemann-Hurwitz formula implies that there are finitely many values of $d, h > 0$ for which there exists a map $\phi$ of degree $d = \deg(\phi)$ from a curve of genus $g$ to a curve of genus $h > 0$, so property (2) is also satisfied for an open dense subset of $\M_{g}$. By construction, it follows that $\M_{g}^{\circ}$ is a fine moduli space \cite[Theorem 2.2.5]{Conrad07} which is representable by a quasiprojective scheme.

For any projective variety $X$ of dimension $n \geq 1$, let $V_{X} \subset \M_{g}^{\circ}$ denote the locus of curves $[C] \in \M_{g}^{\circ}$ such that there exists a quasi-projective variety $Y$ of dimension $n-1$ together with a dominant morphism $C \times Y \rightarrow X$.

\begin{lemma}
For $g \geq 3$, the locus $V_{X}$ is a countable union of constructible subsets of $\M_{g}^{\circ}$.
\end{lemma}

\begin{proof}

Fix an integer $g \geq 3$ and fix a class $\beta \in H_2(X,\Z)$. Let $\overline{\M}_{g}(X, \beta)$ denote the Kontsevich moduli stack of genus $g$ stable maps into $X$ with class $\beta$. There is a forgetful map $\overline{\M}_{g}(X, \beta) \rightarrow \overline{\M}_{g}$, and the universal family $\mathcal{C} \coloneqq \overline{\M}_{g,1}(X, \beta)$ admits an evaluation map which fits into the diagram:
\begin{center}
\begin{tikzcd}
\mathcal{C} \arrow[r, "\textbf{ev}"] \arrow[d] \arrow[rd, swap, "\pi"] & X \\
\overline{\M}_{g}(X, \beta) \arrow[r] & \overline{\M}_{g}
\end{tikzcd}
\end{center}
The locus of closed points $[C] \in \overline{\M}_{g}$ for which the restriction of $\textbf{ev}$ to $\pi^{-1}([C])$ dominates $X$ is by definition constructible, so the intersection with $\M_{g}^{\circ}$ is still constructible in $\M_{g}^{\circ}$. Taking the union over all such $\beta$ and noting that $H_2(X,\Z)$ is countable gives the desired statement.
\end{proof}

Next, we will need a proposition which is inspired by some of the ideas in \cite{Lou24}.

\begin{proposition}\label{lem:XnotcoveredbyC}
Let $X$ be an irreducible non-uniruled variety of dimension $n \geq 1$ over $\C$. If $g \geq \max \{ 3, n+1 \}$, then $V_{X} \subsetneq \M_{g}^{\circ}$ is a countable union of constructible subsets, each of which has dimension strictly less than $\dim \M_{g}^{\circ}$.
\end{proposition}

\begin{proof}
    Without loss of generality, after compactifying and passing to a resolution we may assume that $X$ is smooth and projective.
    
    Suppose for contradiction that some component of $V_{X}$ is Zariski-dense in $\M_{g}^{\circ}$. Since $H_{2}(X, \Z)$ is countable and $\M_{g}$ is uncountable, there exists a class $\beta \in H_2(X,\Z)$ and an irreducible component $\H$ of $\overline{\M}_{g}(X, \beta)$ mapping surjectively onto $\overline{\M}_{g}$ such that the general fiber of $\mathcal{C}_{\H} \rightarrow \overline{\M}_{g}$ (here $\mathcal{C}_{\H}$ is the universal family over $\H$) dominates $X$ under the evaluation map.
    Therefore, $\dim \H \geq 3g-3 + n-1 \geq 3g-3$. The general element of $\H$ corresponds to a morphism $h \colon C \rightarrow X$, and from the definition of $V_{X}$ we see that $h$ extends to a dominant morphism $C \times Y \rightarrow X$, for some quasiprojective variety $Y$ of dimension $n-1$.

    Now let $N_{h}$ denote the normal sheaf, which sits in the exact sequence
    \begin{equation}\label{eq:tangentnormalseq}
    0 \rightarrow T_{C} \rightarrow h^{\ast}T_{X} \rightarrow N_{h} \rightarrow 0.
    \end{equation}
    Recall that the dimension of the Zariski tangent space to any point $[h: C \rightarrow X] \in \overline{\M}_{g}(X, \beta)$ is bounded from above by $h^{0}(N_{h})$ (cf. \cite[proof of Theorem 3.5]{Lou24}). 
    Since $h$ is a general member of $\H$, it follows that $h^{0}(N_{h}) \geq \dim \H$. Next, we will bound $h^{0}(N_{h})$ from above.
    
    \textbf{Claim.} There is an injection of sheaves $\O^{\oplus n-1} \hookrightarrow N_{h}$.
    
    Intuitively, this follows from the fact that $C$ moves in an ($n-1$)-dimensional family. To make this precise, let $C_{y} \cong C$ denote the general member of the family $C \times Y$, which maps to $X$ and let $h \colon C_{y} \rightarrow X$ be the map. We can write down the following diagram of sheaves:
    \begin{center}
    \begin{tikzcd}
    0 \arrow[r] & T_{C_{y}} \arrow[r] \arrow[d, equals] & T_{C \times Y} \big|_{C_{y}} \arrow[r] \arrow[d, "(\ast)"] & N_{C_{y}/C \times Y} \arrow[r] \arrow[d, "(\ast \ast)"] & 0 \\
    0 \arrow[r] & T_{C} \arrow[r] & h^{\ast}T_{X} \arrow[r] & N_{h} \arrow[r] & 0
    \end{tikzcd}
    \end{center}
    Here, the map $(\ast)$ is the restriction of $T_{C \times Y} \rightarrow h^{\ast}T_{X}$ to $C_{y} \cong C$, so it is generically injective. But $T_{C \times Y}|_{C_{y}}$ is locally free (in particular torsion free) so $(\ast)$ is injective, which implies by the Four-Lemma that the map $(\ast \ast)$ is also injective. Since $C_{y}$ is the general fiber of a map, we have the isomorphism $N_{C_{y}/C \times Y} \cong \O_{C_{y}}^{\oplus n-1}$, which gives the desired injection of sheaves.
    
    By comparing ranks, we arrive at an exact sequence of sheaves
    \[ 0 \rightarrow \O^{\oplus n-1} \rightarrow N_{h} \rightarrow \tau \rightarrow 0, \]
    where $\tau$ is a torsion sheaf, i.e., it is supported along a finite set of closed points on $C$. This implies that
    \[ h^{0}(N_{h}) \leq n - 1 + \text{length}(\tau) = n-1 + c_{1}(N_{h}) = n - 1 - (h^{\ast}K_{X}\cdot C) + 2g-2, \]
    where the last equality follows from \eqref{eq:tangentnormalseq}. Since $X$ is non-uniruled and $C$ is a covering family, by 
    \cite[Theorem 0.2 and Corollary 0.3]{BDPP13} we know that $(h^{\ast}K_{X} \cdot C) \geq 0$. Since $g > n$, this contradicts the fact that $h^{0}(N_{h}) \geq 3g - 3$.  
\end{proof}

Now we are ready to show:

\begin{proof}[Proof of Theorem~\ref{thm:BirX=BirXxC}]
    Choose a curve $[C] \in \M_{g}^{\circ} \setminus V_{X}$, which is possible due to Proposition~\ref{lem:XnotcoveredbyC}. 
    Let $\varphi \in \Bir(X \times C)$ be a birational transformation and let $\pi_{1}, \pi_{2}$ be the projection maps to $X$ and $C$, respectively. It suffices to show that $\varphi$ is of the form $(\psi, \id_{C})$ for some $\psi \in \Bir(X)$.
    
    First, we claim that for a general point $x \in X$, the image $(\pi_{1} \circ \varphi)(\{ x \} \times C)$ is a point in $X$. Suppose for contradiction that this is not the case. Since $C$ is a general curve of genus $g \geq 2$, it does not admit any maps of degree $\geq 2$ to any curve other than $\P^{1}$. The map $\varphi$ is dominant and $X$ is not uniruled, so the restriction $(\pi_{1} \circ \varphi)|_{\{ x \} \times C}$ must then be birational onto its image. Now setting $Y$ to be a general hyperplane section of $X$, we see that $\pi_{1} \circ \varphi|_{Y \times C}$ defines a dominant rational map $Y \times C \dashrightarrow X$, which contradicts our choice of $[C]$.

    Thus, for a general point $x \in X$, the image $(\pi_{1} \circ \varphi)(\{ x \} \times C)$ is indeed a point in $X$. By rigidity, there exists a map $\psi \colon X \dashrightarrow X$ such that the diagram below commutes:
    \begin{center}
    \begin{tikzcd}
        X \times C \arrow[r, dashed, "\varphi"] \arrow[d, swap, "\pi_{1}"] & X \times C \arrow[d, "\pi_{1}"] \\
        X \arrow[r, dashed, "\psi"] & X
    \end{tikzcd}
    \end{center}
    Since $\varphi$ is birational, the map $\psi$ must also be birational. Furthermore, since $C$ is general we see that $\Aut(C) = \{ 1 \}$. This implies that $\varphi = (\psi, \id_{C})$, which is what we want.
\end{proof}

\subsection{Examples and Questions}

Combining \Cref{thm:intromain} and \Cref{thm:introBirX=BirXxC}, the only remaining case
is uniruled but non-ruled varieties. Note that of course there are uniruled varieties with trivial birational transformation group. In fact, many of these satisfy the same property as in the theorem above.

\begin{example}
Let $X$ be a smooth variety such that $\Bir(X)$ is discrete and $H^{1}(X,\O_{X}) = 0$ (e.g. any birationally superrigid Fano variety such as a smooth index 1 Fano hypersurface $X_{d} \subset \P^{n+1}$, where $d = n+1$ and $n \geq 3$). Let $C$ be any curve of genus $g \geq 3$ with $\Aut(C) = \{ 1 \}$. Since $H^{1}(X, \O_{X}) = 0$, there are no non-constant rational maps from $X$ to $C$. Thus, $\Bir(X \times C) = \Bir(X\times C, \mathrm{pr}_2) \cong \Bir(X)$, where the last isomorphism follows from the fact that $\Bir(X)$ is discrete and hence any birational transformation of $X \times C$ restricts to the same birational transformation of $X$ as we vary $p \in C$.
\end{example}

However, it is plausible that some non-ruled but uniruled 
varieties with large $\Bir(X)$ could be determined (up to birational equivalence
and field automorphism) by their birational 
transformation groups.  Nevertheless, proving this seems much more difficult 
since in this case
$\Bir(X)$ does not contain nontrivial connected algebraic subgroups.

\begin{example}
Let $X$ be a smooth cubic threefold.  Then $X$ is uniruled (in fact 
unirational) but not ruled.  Cantat has asked whether there exists some
$Y$, not necessarily a cubic threefold, such that $\Bir(X) \cong \Bir(Y)$ as groups, but the fields $\C(X)$ and $\C(Y)$ are not abstractly isomorphic. We remark that the strategy employed in \Cref{thm:BirX=BirXxC} does not work
in this case.  Indeed, given any point $x \in X$, there is a birational 
involution on $X$ defined at a general point $y \in X$ by sending $y$ to the third intersection point of $X$ with the line spanned by $x$ and $y$.  In 
particular, there is an injective algebraic family of birational transformations
$X \hookrightarrow \Bir(X)$. Now choose any line $\ell$ in $X$.  For any 
positive-dimensional
variety $V$, a nonconstant rational function on $V$ gives a dominant rational map 
$V \dashrightarrow \mathbb{P}^1 \cong \ell$, which we can compose with
$\ell \subset X \rightarrow \Bir(X)$ to get a nontrivial family of birational
transformations parametrized by $V$.  This proves that the natural injection
$\Bir(X) \hookrightarrow \Bir(X \times V)$ is not an isomorphism.
\end{example}

In summary, the question of whether $\Bir(X)$ determines $X$ (up to birational transformation
and field automorphism) among all varieties is answered positively for ruled 
varieties by \Cref{thm:intromain} and negatively for non-uniruled varieties by
\Cref{thm:introBirX=BirXxC}.  As discussed in the previous section, the situation
for uniruled but non-ruled varieties is less clear, with a variety of possible behaviors.  We pose the following question, which we hope to investigate further 
in the future.

\begin{question}
\label{quest:nonruled}
For which uniruled non-ruled varieties $X$ over $\C$ does the group $\Bir(X)$ characterize $X$ up to birational transformations and automorphisms of the base field?
\end{question}

Even when \Cref{thm:intromain} applies, the geometric 
relationship between two ruled varieties $V_1$ and $V_2$ with isomorphic birational transformation groups is not entirely clear.  The theorem shows that
some birational models of $V_1$ and $V_2$ are isomorphic as schemes
over $\Z$.  In specific
cases, such as when $V_1$ is rational, we can conclude that $V_1$ and $V_2$
are actually birationally equivalent, since the property of the function
field being a purely transcendental extension of $\C$ is preserved by base change
by a $\C$-automorphism.  In the more general case, we can ask:

\begin{question}
\label{quest:conj_inv}
For which varieties does base change by a field automorphism of $\C$ not change the birational type?
\end{question}

By \Cref{thm:intromain}, any \textit{ruled} variety with this property is actually determined up to birational 
equivalence by its group of birational transformations.

A variety $X$ satisfies the property in \Cref{quest:conj_inv} if $X$ is birational
to a $\C$-variety is defined over, say, $\Q$.
However, the property does not always hold.
For instance, take
two very general elliptic curves $E_i, i = 1,2$ and let
$V_i = \mathbb{P}^1 \times E_i$.  Then $V_1$ and $V_2$ are ruled varieties 
isomorphic as schemes over $\Z$ and with isomorphic birational transformation
groups.  Nevertheless, $V_1$ and $V_2$ are not birationally equivalent.

\bibliographystyle{amsalpha}
\bibliography{biblio}

\end{document}